\documentclass[12pt,reqno]{amsproc}%
\usepackage{amsfonts}
\usepackage{amsmath}
\usepackage{amssymb}
\usepackage{graphicx}%
\providecommand{\U}[1]{\protect \rule{.1in}{.1in}}
 \theoremstyle{plain}
\newtheorem{theorem}{Theorem}[section]

\newtheorem{corollary}[theorem]{Corollary}

\newtheorem{definition}[theorem]{Definition}
\newtheorem{example}[theorem]{Example}

\newtheorem{lemma}[theorem]{Lemma}

\newtheorem{proposition}[theorem]{Proposition}
\newtheorem{remark}[theorem]{Remark}

\numberwithin{equation}  {section}
\begin{document}
\title[Beurling Theorem]{A General Beurling-Helson-Lowdenslager Theorem on the Disk}
\author{Yanni Chen }
\address{Department of Mathematics, University of New Hampshire, Durham, NH 03824, U.S.A.}
\email{unhcyn@gmail.com}
\thanks{*Corresponding author email address: unhcyn@gmail.com }

\begin{abstract}
The classical Beurling-Helson-Lowdenslager theorem characterizes the
shift-invariant subspaces of the Hardy space $H^{2}$ and of the Lebesgue space $L^{2}$. In this
paper, which is self-contained, we define a very general class of norms
$\alpha$ and define spaces $H^{\alpha}$ and $L^{\alpha}.$ We then extend the
Beurling-Helson-Lowdenslager invariant subspace theorem.  The idea of the
proof is new and quite simple; most of the details involve extending basic
well-known $\left \Vert {}\right \Vert _{p}$-results for our more general norms.

\end{abstract}
\maketitle

\section{Introduction}

Why should we care about invariant subspaces? In finite dimensions all of the
structure theorems for operators can be expressed in terms of invariant
subspaces. For example the statement that every $n\times n$ complex matrix $T$
is unitarily equivalent to an upper triangular matrix is equivalent to the
existence of a chain $M_{0}\subset M_{1}\subset \cdots \subset M_{n}$ of
$T$-invariant linear subspaces with $\dim M_{k}=k$ for $0\leq k\leq n$. Since
every upper triangular normal matrix is diagonal, the preceding result yields
the spectral theorem. A matrix is similar to a single Jordan block if and only
if its set of invariant subspaces is linearly ordered by inclusion, so the
Jordan canonical form can be completely described in terms of invariant
subspaces. In \cite{BF} L. Brickman and P. A. Fillmore describe the lattice of
all invariant subspaces of an arbitrary matrix.

In infinite dimensions, where we consider closed subspaces and bounded
operators, even the existence of one nontrivial invariant subspace remains an
open problem for Hilbert spaces. If $T$ is a normal operator with a $\ast
$-cyclic vector, then, by the spectral theorem, $T$ is unitarily equivalent to
the multiplication operator $M_{z}$ on $L^{2}\left(  \sigma \left(  T\right)
,\mu \right)  ,$ i.e.,
\[
\left(  M_{z}f\right)  \left(  z\right)  =zf\left(  z\right)  ,
\]
where $\mu$ is a probability Borel measure on the spectrum $\sigma \left(
T\right)  $ of $T.$ In this case von Neumann proved that if a subspace $W$
that is invariant for $M_{z}$ and for $M_{z}^{\ast}=M_{\bar{z}}$, then the
projection $P$ onto $W$ is in the commutant of $M_{z},$ which is the maximal
abelian algebra $\left \{  M_{\varphi}:\varphi \in L^{\infty}\left(  \mu \right)
\right \}  $. Hence there is a Borel subset $E$ of $\sigma \left(  T\right)  $
such that $P=M_{\chi_{E}}$, which implies $W=\chi_{E}L^{2}\left(  \mu \right)
$. It follows that if $T$ is a \emph{reductive} normal operator, i.e., every
invariant subspace for $T$ is invariant for $T^{\ast}$, then all invariant
subspaces of $T$ have the form $\chi_{E}L^{2}\left(  \mu \right)  .$ In
\cite{S} D. Sarason characterized the $\left(  M_{z},\mu \right)  $ that are
reductive; in particular, when $T=M_{z}$ is unitary (i.e., $\sigma \left(
T\right)  \subset \mathbb{T}=\left \{  \lambda \in \mathbb{C}:\left \vert
\lambda \right \vert =1\right \}  $), then $M_{z}$ is reductive if and only if
Haar measure $m$ on $\mathbb{T}$ is not absolutely continuous with respect to
$\mu$. When $\sigma \left(  T\right)  =\mathbb{T}$ and $\mu=m$ is Haar measure
on $\mathbb{T}$, then $M_{z}$ on $L^{2}$ is the bilateral shift operator.

If $T$ is the restriction of a normal operator to an invariant subspace with a
cyclic vector $e$, then there is a probability space $L^{2}\left(
\sigma \left(  T\right)  ,\mu \right)  $ such that $T$ is unitarily equivalent
to $M_{z}$ restricted to $P^{2}\left(  \mu \right)  $, the closure of the
polynomials in $L^{2}\left(  \mu \right)  $, and where $e$ corresponds to the
constant function $1\in P^{2}\left(  \mu \right)  $. If $\sigma \left(
T\right)  =\mathbb{T}$ and $\mu=m$, then $P^{2}\left(  \mu \right)  $ is the
classical Hardy space $H^{2}$ and $M_{z}$ is the unilateral shift operator.

In infinite dimensions the first important characterization of all the
invariant subspaces of a non-normal operator, the unilateral shift, was due to
A. Beurling \cite{B} in 1949. His result was extended by H. Helson and D.
Lowdenslager \cite{HL} to the bilateral shift operator, which is a
non-reductive unitary operator.

In this paper, suppose $\mathbb{D}$ is the unit disk in the complex plane
$\mathbb{C},$ and $m$ is Haar measure (i.e., normalized arc length) on the
unit circle $\mathbb{T}=\left \{  \lambda \in \mathbb{C}:\left \vert
\lambda \right \vert =1\right \}  $. We let $\mathbb{R},$ $\mathbb{Z}$,
$\mathbb{N}$, respectively, denote the sets of real numbers, integers, and
positive integers. Since $\left \{  z^{n}:n\in \mathbb{Z}\right \}  $ is an
orthonormal basis for $L^{2},$ we see that $M_{z}$ is a \emph{bilateral shift}
operator. The subspace $H^{2}$ which is the closed span of $\left \{
z^{n}:n\geq0\right \}  $ is invariant for $M_{z}$ and the restriction of
$M_{z}$ to $H^{2}$ is a \emph{unilateral shift} operator. A closed linear
subspace $W$ of $L^{2}$ is \emph{doubly invariant} if $zW\subseteq W$ and
$\bar{z}W\subseteq W$. Since $\bar{z}z=1$ on $\mathbb{T}$, $W$ is doubly
invariant if and only if $zW=W$. Since the set of polynomials in $z$ and
$\bar{z}$ is weak*-dense in $L^{\infty}=L^{\infty}\left(  \mathbb{T}\right)
,$ and since the weak* topology on $L^{\infty}$ coincides with the weak
operator topology on $L^{\infty}$ (acting as multiplication operators on
$L^{2}$), $W$ is doubly invariant if and only if $L^{\infty}\mathbb{\cdot
}W\subseteq W$. A subspace $W$ is \emph{simply invariant} if $zW\subsetneqq
W$, which means $H^{\infty}\cdot W\subseteq W$, but $\bar{z}W\nsubseteq W$.

We state the classical Beurling-Helson-Lowdenslager theorem for a closed
subspace $W$ of $L^{2}$. A very short elegant proof is given in \cite{N}. We
give a short proof to make this paper self-contained.

\begin{theorem}
\textbf{(Beurling-Helson-Lowdenslager)\label{BHL} }Suppose $W$ is a closed
linear subspace of $L^{2}$ and $zW\subset W$. Then

\begin{enumerate}
\item if $W$ is doubly invariant, then $W=\chi_{E}L^{2}$ for some Borel subset
$E$ of $\mathbb{T};$

\item if $W\ $is simply invariant, then $W=\varphi H^{2}$ for some $\varphi \in
L^{\infty}$ with $\left \vert \varphi \right \vert =1$ a.e. $\left(  m\right)  ;$

\item if $0\neq W\subseteq H^{2}$, then $W=\varphi H^{2}$ with $\varphi$ an
inner function (i.e., $\varphi \in H^{\infty}$ and $\left \vert \varphi
\right \vert =1$ a.e. $\left(  m\right)  $).
\end{enumerate}
\end{theorem}

\begin{proof}
1. This follows from von Neumann's result discussed above.

2. If $W$ is simply invariant, then $M_{z}|W$ is a nonunitary isometry, which,
by the Halmos-Wold-Kolmogorov decomposition must be a direct sum of at least
one unilateral shift and an isometry. Thus $W=W_{1}\oplus W_{2}$ and there is
a unit vector $\varphi \in W_{1}$ with $\left \{  z^{n}\varphi:n\geq0\right \}  $
an orthonormal basis for $W_{1}$. Since $\varphi \bot z^{n}\varphi$ for all
$n\geq1,$ we have%
\[
\int_{\mathbb{T}}\left \vert \varphi \right \vert ^{2}z^{n}dm=0
\]
for all $n\geq1$, and taking conjugates, we also get the same result for all
$n\leq-1,$ which implies that $|\varphi(z)|^{2}=\sum_{n=-\infty
}^{\infty}c_{n}z^{n}=c_{0}.$ Since $\varphi$ is a unit vector, we have
$|\varphi|^{2}=1$ a.e. $\left(  m\right)  $. Hence,
\[
W_{1}=\varphi \cdot \overline{sp}\left(  \left \{  z^{n}:n\geq0\right \}  \right)
=\varphi H^{2}.
\]
If $g$ is a unit vector in $W_{2},$ then we have $z^{n}\varphi$ $\bot$ $g$ and
$\varphi$ $\bot$ $z^{n}g$ for $n\geq0,$ which yields%
\[
\int_{\mathbb{T}}z^{n}\varphi \bar{g}dm=0
\]
for all $n\in \mathbb{Z}$. It follows from the definition of Fourier
coefficients of $\varphi g$ that $\left \vert \varphi \bar{g}\right \vert =0,$
which implies $|g|=|\varphi \bar{g}|=0.$ Hence $W=W_{1}=\varphi H^{2}$ .

3. It is clear that no nonzero subspace $W\supseteq zW$ of $H^{2}$ can have
the form $\chi_{E}L^{2}$ , so the only possibility is the situation in part
$(2),$ which means $W=\varphi H^{2}.$ Also, $\varphi \in \varphi H^{2}$
$=W\subset H^{2},$ which implies $\varphi \in H^{2}$ is inner.
\end{proof}

These results are also true when $\left \Vert {}\right \Vert _{2}$ is replaced
with $\left \Vert {}\right \Vert _{p}$ for $1\leq p\leq \infty,$ with the
additional assumption that $W$ is weak*-closed when $p=\infty$ (see
\cite{Sr1}, \cite{Sr2}, \cite{Sr3}). Many of the proofs for the $\left \Vert
{}\right \Vert _{p}$ case use the $L^{2}$ result and take cases when $p\leq2$
and $2<p$. In \cite{Chen} and \cite{Chen2}, the author proved version of
part $\left(  3\right)  $ for a large class
of norms, called \emph{rotationally invariant} norms. In these more general
setting, the cases $p\leq2$ and $2<p$ have no analogue.

In this paper we extend the Beurling-Helson-Lowdenslager theorem to an even
larger class of norms, the $\emph{continuous}$ $\Vert \Vert_{1}$%
\emph{-dominating normalized gauge norms}, with a proof that is simple even in
the $L^{p}$ case. For each such norm $\alpha$ we define a Banach space
$L^{\alpha}$ and a Hardy space $H^{\alpha}$ with
\[
L^{\infty}\subset L^{\alpha}\subset L^{1}\text{ and }H^{\infty}\subset
H^{\alpha}\subset H^{1}.
\]
In this new setting, we prove the following Beurling-Helson-Lowdenslager
theorem, which is the main result of this paper.\\

\textbf{Theorem 3.6.} Suppose $\alpha$ is a continuous $\Vert \Vert_{1}%
$-dominating normalized gauge norm and $W$ is a closed subspace of $L^{\alpha
}$. Then $zW\subseteq W$ if and only if either $W=\phi H^{\alpha}$ for some
unimodular function $\phi$ or $W=\chi_{E}L^{\alpha}$ for some Borel set
$E\subset \mathbb{T}$. If $0\neq W\subset H^{\alpha}$, then $W=\varphi
H^{\alpha}$ for some inner function $\varphi$.\newline

To prove Theorem 3.6. we need the following technical theorem in Section
3.\newline

\textbf{Theorem 3.4.} Suppose $\alpha$ is a continuous $\Vert \Vert_{1}%
$-dominating normalized gauge norm. Let $W$ be an $\alpha$-closed linear
subspace of $L^{\alpha}$, and $M$ be a weak*-closed linear subspace of
$L^{\infty}$ such that $zM\subseteq M$ and $zW\subseteq W$. Then

\begin{enumerate}
\item $M=[M]^{-\alpha}\cap L^{\infty};$

\item $W\cap M$ is weak*-closed in $L^{\infty};$

\item $W=[ W\cap L^{\infty}] ^{-\alpha}.$
\end{enumerate}

This gives us a quick route from the $\left \Vert {}\right \Vert _{2}$-version
of invariant subspace structure  to the (weak*-closed) $\left \Vert
{}\right \Vert _{\infty}$-version of invariant subspace structure, and then to
the $\alpha$-version of invariant subspace structure.

\section{Preliminaries}

\begin{definition}
A norm $\alpha$ on $L^{\infty}$ is called a $\Vert \Vert_{1}$\emph{-dominating
normalized gauge norm} if

\begin{enumerate}
\item $\alpha(1)=1,$

\item $\alpha(|f|)=\alpha(f)$ for every $f\in L^{\infty},$

\item $\alpha(f)\geq \|f\|_{1}$ for every $f\in L^{\infty}.$
\end{enumerate}
\end{definition}

We say that a $\Vert \Vert_{1}$-dominating normalized gauge norm is
$continuous$ if
\[
\lim_{m(E)\rightarrow0^{+}}\alpha(\chi_{E})=0.
\]

Although a $\Vert \Vert_{1}$-dominating normalized gauge norm $\alpha$ is
defined only on $L^{\infty},$ we can define $\alpha$ for all measurable
functions $f$ on $\mathbb{T}$ by
\[
\alpha(f)=\sup \{ \alpha(s):s\text{ is a simple function},0\leq s\leq|f|\}.
\]
It is clear that $\alpha \left(  f\right)  =\alpha \left(  \left \vert
f\right \vert \right)  $ still holds.

We define
\[
{\mathcal{L}}^{\alpha}=\{f:\alpha(f)<\infty \},
\]
and define $L^{\alpha}$ to be the $\alpha$-closure of $L^{\infty}$ in
$\mathcal{L}^{\alpha}$.

\begin{lemma}
\label{property of alpha}  Suppose $f,g:\mathbb{T}\rightarrow \mathbb{C}$ are
measurable. Let $\alpha$ be a $\Vert \Vert_{1}$-dominating normalized gauge
norm. Then the following statements are true:

\begin{enumerate}
\item $|f|\leq|g|\Longrightarrow \alpha(f)\leq \alpha(g);$

\item $\alpha(fg)\leq \alpha(f)\Vert g\Vert_{\infty};$

\item $\alpha(g)\leq \Vert g\Vert_{\infty};$

\item $L^{\infty}\subset L^{\alpha}\subset \mathcal{L}^{\alpha}\subset L^{1};$

\item If $\alpha$ is continuous,  $0\leq f_{1}\leq f_{2}\leq \cdots$ and
$f_{n}\rightarrow f$ a.e. $\left(  m\right)  $, then $\alpha \left(
f_{n}\right)  \rightarrow \alpha \left(  f\right)  ;$

\item If $\alpha$ is continuous, then $\mathcal{L}^{\alpha}$ and $L^{a}$ are
both Banach spaces.
\end{enumerate}
\end{lemma}

\begin{proof}
1. It is clear that if $\left \vert u\right \vert =1$ a.e. $\left(  m\right)  $,
then
\[
\alpha \left(  uf\right)  =\alpha \left(  \left \vert uf\right \vert \right)
=\alpha \left(  \left \vert f\right \vert \right)  =\alpha \left(  f\right)  .
\]
If $\left \vert f\right \vert \leq \left \vert g\right \vert ,$ then there is a
measurable $h$ with $\left \vert h\right \vert \leq1$ a.e. $\left(  m\right)  $
such that $f=hg$. Then there are measurable functions $u_{1}$ and $u_{2}$ with
$\left \vert u_{1}\right \vert =\left \vert u_{2}\right \vert =1$ a.e. $\left(
m\right)  $ with $h=\left(  u_{1}+u_{2}\right)  /2$. Hence
\[
\alpha \left(  f\right)  =\alpha \left(  \left(  u_{1}g+u_{2}g\right)
/2\right)  \leq \frac{1}{2}\left[  \alpha \left(  u_{1}g\right)  +\alpha \left(
u_{2}g\right)  \right]  =\alpha \left(  g\right)  .
\]

2. This follows from part (1) and the fact that $\left \vert fg\right \vert
\leq \left \vert f\right \vert \left \Vert g\right \Vert _{\infty}$ a.e. $\left(
m\right)  $.

3. This follows from part (2) with $f=1$.

4. Since $\left \Vert s\right \Vert _{1}\leq \alpha \left(  s\right)  $ whenever
$s\in L^{\infty},$ it follows for any measurable $f$ that
\[
\alpha \left(  f\right)  \geq \sup \{ \left \Vert s\right \Vert _{1}:s\text{ is a
simple function},0\leq s\leq|f|\}=\left \Vert f\right \Vert _{1}.
\]
This implies $\mathcal{L}^{\alpha}\subset L^{1}$. The inclusions $L^{\infty
}\subset L^{\alpha}\subset \mathcal{L}^{\alpha}$ follow from the definition of
$L^{\alpha}$.

5. Suppose $0\leq s\leq f$ and $0\leq t<1.$ Write $s=\sum_{1\leq k\leq m}%
a_{k}\chi_{E_{k}}$ with $0<a_{k}$ for $1\leq k\leq m$ and $\left \{
E_{1},\ldots,E_{m}\right \}  $ disjoint. If we let $E_{k,n}=\left \{  \omega \in
E_{k}:ta_{k}<f_{n}\left(  \omega \right)  \right \}  ,$ we see that%
\[
E_{k,1}\subset E_{k,2}\subset \cdots \ \ \text{ and}\ \ \cup_{1\leq n<\infty}%
E_{k,n}=E_{k}.
\]
Since $\alpha$ is continuous,
\[
\alpha \left(  \chi_{E_{k}}-\chi_{E_{k,n}}\right)  =\alpha \left(  \chi
_{E_{k}\backslash E_{k,n}}\right)  \rightarrow0.
\]
Hence%
\[
t\alpha \left(  s\right)  =\lim_{n\rightarrow \infty}\alpha \left(  \sum
_{k=1}^{m}ta_{k}\chi_{E_{k,n}}\right)  \leq \lim_{n\rightarrow \infty}%
\alpha \left(  f_{n}\right)  .
\]
Since $t$ was arbitrary, for every simple function $s$ with $0\leq s\leq f,$
we have%
\[
\alpha \left(  s\right)  \leq \lim_{n\rightarrow \infty}\alpha \left(
f_{n}\right)  .
\]
By the definition of $\alpha \left(  f\right)  $, we see that $\alpha \left(
f\right)  \leq \lim_{n\rightarrow \infty}\alpha \left(  f_{n}\right)  $, and
$\alpha \left(  f_{n}\right)  \leq \alpha \left(  f\right)  $ for each $n\geq1$
follows from part $\left(  1\right)  ,$ which implies $\lim_{n\rightarrow
\infty}\alpha \left(  f_{n}\right) \leq \alpha(f).$ This completes the proof.

6. It follows from the definition of $\mathcal{L}^{\alpha}$ that
$\mathcal{L}^{\alpha}$ is a normed space with respect to $\alpha.$ To prove
the completeness, suppose $\left \{  f_{n}\right \}  $ is a sequence in
$\mathcal{L}^{\alpha}$ with $\sum_{n=1}^{\infty}\alpha \left(  f_{n}\right)
<\infty.$ Since $\left \Vert {}\right \Vert _{1}\leq \alpha \left(  {}\right)  , $
we know that $g=\sum_{n=1}^{\infty}\left \vert f_{n}\right \vert \in L^{1}$ and
$f=\sum_{n=1}^{\infty}f_{n}$ converges a.e. $\left(  m\right)  $. It follows
from part (5) that
\[
\alpha \left(  g\right)  =\lim_{N\rightarrow \infty}\alpha \left(  \sum_{n=1}%
^{N}\left \vert f_{n}\right \vert \right)  \leq \lim_{N\rightarrow \infty}%
\sum_{n=1}^{N}\alpha \left(  | f_{n}|\right)  =\sum_{n=1}^{\infty}\alpha \left(
f_{n}\right)  <\infty.
\]
Since $\left \vert f\right \vert \leq g,$ from part $(1)$ we know $\alpha \left(
f\right)  <\infty$. Thus $f\in \mathcal{L}^{\alpha}$. Also, by part (1) and
part (5), we have for each $N\geq1,$%
\[
\alpha \left(  f-\sum_{n=1}^{N}f_{n}\right)  \leq \alpha \left(  \sum
_{n=N+1}^{\infty}\left \vert f_{n}\right \vert \right)  \leq \sum_{n=N+1}%
^{\infty}\alpha \left(  f_{n}\right)  =\sum_{n=1}^{\infty}\alpha \left(
f_{n}\right)  -\sum_{n=1}^{N}\alpha \left(  f_{n}\right)  .
\]
Hence
\[
\lim_{N\rightarrow \infty}\alpha \left(  f-\sum_{n=1}^{N}f_{n}\right)  =0.
\]
Since every absolutely convergent series in $\mathcal{L}^{\alpha}$ is
convergent in $\mathcal{L}^{\alpha}$, we know $\mathcal{L}^{\alpha}$ is complete.
\end{proof}

We let $\mathcal{N}$ denote the set of all $\Vert \Vert_{1}$-dominating
normalized gauge norms, and we let $\mathcal{N}_{c}$ denote the set of
continuous ones. We can give $\mathcal{N}$ the topology of pointwise convergence.

\begin{lemma}
The sets $\mathcal{N}$ and $\mathcal{N}_{c}$ are convex, and the set
$\mathcal{N}$ is compact in the topology of pointwise convergence.
\end{lemma}

\begin{proof}
It directly follows from the definitions that $\mathcal{N}$ and $\mathcal{N}%
_{c}$ are convex. Suppose $\left \{  \alpha_{\lambda}\right \}  $ is a net in
$\mathcal{N}$ and choose a subnet $\left \{  \alpha_{\lambda_{\kappa}}\right \}
$ that is an ultranet. Then, for every $f\in L^{\infty}$, we have that
$\left \{  \alpha_{\lambda_{\kappa}}\left(  f\right)  \right \}  $ is an
ultranet in the compact set $\left[  \left \Vert f\right \Vert _{1},\left \Vert
f\right \Vert _{\infty}\right]  .$ Thus
\[
\alpha \left(  f\right)  =\lim_{k}\alpha_{\lambda_{k}}\left(  f\right)
\]
exists. It follows directly from the definition of $\mathcal{N}$ that
$\alpha \in \mathcal{N}$. Hence $\mathcal{N}$ is compact.
\end{proof}

\begin{example} There are many interesting examples of continuous $\|\|_1$-dominating normalized gauge norms other than the usual $\|\|_p$ norms ( $1\leq p< \infty$).
\begin{enumerate}\item If $1\leq p_{n}<\infty$ for $n\geq1,$ then $\alpha=\sum_{n=1}^{\infty}\frac
{1}{2^{n}}\left \Vert {}\right \Vert _{p_{n}}\in \mathcal{N}_{c}$. Moreover if
$p_{n}\rightarrow \infty$, then $\alpha$ is not equivalent to any $\left \Vert
{}\right \Vert _{p}$ .
\item The Lorentz, Marcinkiewicz and Orlicz norms are important examples (e.g., see \cite{AKS}).

\end{enumerate}
\end{example}

\begin{definition}
\label{definition of dual norm} Let $\alpha$ be a $\Vert \Vert_{1}$-dominating
normalized gauge norm. We define the \emph{dual norm} $\alpha^{\prime}:
L^{\infty}\rightarrow[0,\infty]$ by%
\begin{align*}
\alpha^{\prime}\left(  f\right)   &  =\sup \left \{  \left \vert \int
_{\mathbb{T}}fhdm\right \vert :h\in L^{\infty},\alpha \left(  h\right)
\leq1\right \} \\
&  =\sup \left \{  \int_{\mathbb{T}}\left \vert fh\right \vert dm:h\in L^{\infty
},\alpha \left(  h\right)  \leq1\right \}  .
\end{align*}

\end{definition}

\begin{lemma}
\label{property of alpha prime} Let $\alpha$ be a $\Vert \Vert_{1}$-dominating
normalized gauge norm. Then the dual norm $\alpha^{\prime}$ is also a
$\Vert \Vert_{1}$-dominating normalized gauge norm.
\end{lemma}

\begin{proof}
Suppose $f\in L^{\infty}.$ If $h\in L^{\infty}$ with $\alpha(h)\leq1,$ then
\[
\int_{\mathbb{T}}|fh|dm\leq \Vert f\Vert_{\infty}\Vert h\Vert_{1}\leq \Vert
f\Vert_{\infty}\alpha(h)\leq \Vert f\Vert_{\infty},
\]
thus $\alpha^{\prime}(f)\leq \Vert f\Vert_{\infty}.$ On the other hand, since
$\alpha \left(  1\right)  =1,$ we have
\[
\alpha^{\prime}\left(  f\right)  \geq \int_{\mathbb{T}}|f|1dm=\Vert f\Vert
_{1},
\]
it follows that $\alpha^{\prime}$ is a norm with $\alpha^{\prime}%
(|f|)=\alpha^{\prime}(f)$ for every $f\in L^{\infty}.$ Since $\left \Vert
1\right \Vert _{1}\leq \alpha^{\prime}\left(  1\right)  \leq \left \Vert
1\right \Vert _{\infty}$, we see that $\alpha^{\prime}(1)=1.$
\end{proof}

Now we are ready to describe the dual space of $L^{\alpha},$ when $\alpha$ is
a continuous $\Vert \Vert_{1}$-dominating normalized gauge norm.

\begin{proposition}
\label{dualspace} Let $\alpha$ be a continuous $\Vert \Vert_{1}$-dominating
normalized gauge norm and let $\alpha^{\prime}$ be the dual norm of $\alpha$
as in Definition \ref{definition of dual norm}. Then $\left(  L^{\alpha
}\right)  ^{\sharp}={\mathcal{L}}^{\alpha^{\prime}}$, i.e., for every $\phi
\in \left(  L^{\alpha}\right)  ^{\sharp}$, there is an $h\in \mathcal{L}%
^{\alpha^{\prime}}$ such that $\left \Vert \phi \right \Vert =\alpha^{\prime
}\left(  h\right)  $ and
\[
\phi(f)=\int_{\mathbb{T}}fhdm\
\]
for all $f\in L^{\alpha}.$
\end{proposition}

\begin{proof}
If $\left \{  E_{n}\right \}  $ is a countable collection of disjoint  Borel subsets of
$\mathbb{T}$, it follows that
\[
\lim_{N\rightarrow \infty}m\left(  \cup_{n=N+1}^{\infty}E_{n}\right)  =0,
\]
and the continuity of $\alpha$ implies
\[
\lim_{N\rightarrow \infty}\left \vert \phi \left(  \sum_{n=N+1}^{\infty}%
\chi_{E_{n}}\right)  \right \vert \leq \lim_{N\rightarrow \infty}\left \Vert
\phi \right \Vert \alpha \left(  \sum_{n=N+1}^{\infty}\chi_{E_{n}}\right)  =0.
\]
Hence%
\[
\phi \left(  \sum_{n=1}^{\infty}\chi_{E_{n}}\right)  =\sum_{n=1}^{\infty}%
\phi \left(  \chi_{E_{n}}\right)  .
\]
It follows that the restriction of $\phi$ to $L^{\infty}$ is weak*-continuous,
which implies there is an $h\in L^{1}$ such that, for every $f\in L^{\infty},$%
\[
\phi \left(  f\right)  =\int_{\mathbb{T}}fhdm.
\]
Since $L^{\infty}$ is
dense in $L^{\alpha},$ it follows that%
\[
\phi \left(  f\right)  =\int_{\mathbb{T}}fhdm
\]
holds for all $f\in L^{\alpha}$. Moreover, the definitions of $\alpha^{\prime}$ and $\left \Vert \phi \right \Vert $ imply
that $\alpha^{\prime}(h)=\left \Vert \phi \right \Vert .$
\end{proof}

\bigskip

\begin{corollary}
Let $\alpha$ be a continuous $\Vert \Vert_{1}$-dominating normalized gauge norm
and let $\alpha^{\prime}$ be the dual norm of $\alpha$ as in Definition
\ref{definition of dual norm}. Then $\mathcal{L}^{\alpha^{\prime}}$ is a
Banach space.
\end{corollary}

\begin{proof}
It follows from Theorem \ref{dualspace} that $\mathcal{L}^{\alpha^{\prime}}$
is the dual space of $L^{\alpha},$ thus $\mathcal{L}^{\alpha^{\prime}}$ is a
Banach space.
\end{proof}

We let $\mathbb{B}=\left \{  f\in L^{\infty}:\left \Vert f\right \Vert _{\infty
}\leq1\right \}  $ denote the closed unit ball in $L^{\infty}$.

\begin{lemma}
\label{property of unit ball} \label{=top}Let $\alpha$ be a continuous
$\Vert \Vert_{1}$-dominating normalized gauge norm. Then

\begin{enumerate}
\item The $\alpha$-topology and the $\left \Vert {}\right \Vert _{2}$-topology
coincide on $\mathbb{B}$.

\item $\mathbb{B}=\left \{  f\in L^{\infty}:\left \Vert f\right \Vert _{\infty
}\leq1\right \}  $ is $\alpha$-closed.
\end{enumerate}
\end{lemma}

\begin{proof}
1. Since $\alpha$ is $\Vert \Vert_{1}$-dominating, $\alpha$-convergence implies
$\left \Vert {}\right \Vert _{1}$-convergence, which implies convergence in
measure. Suppose $\left \{  f_{n}\right \}  $ is a sequence in $\mathbb{B}$ and
$f\in \mathbb{B}$ with $f_{n}\rightarrow f$ in measure and $\varepsilon>0.$ If
$E_{n}=\left \{  z\in \mathbb{T}:\left \vert f\left(  z\right)  -f_{n}\left(
z\right)  \right \vert \geq \frac{\varepsilon}{2}\right \}  ,$ then $\lim
_{n\rightarrow \infty}m\left(  E_{n}\right)  =0.$ Since $\alpha$ is continuous,
we have $\lim_{n\rightarrow \infty}\alpha \left(  \chi_{E_{n}}\right)  =0,$
which implies that%

\begin{align}
\alpha \left(  f_{n}-f\right)  & =\alpha((f-f_{n})\chi_{E_{n}}+(f-f_{n}%
)\chi_{\mathbb{T}\backslash E_{n}})\nonumber \\
& \leq \alpha((f-f_{n})\chi_{E_{n}})+\alpha((f-f_{n})\chi_{\mathbb{T}\backslash
E_{n}})\nonumber \\
& < \alpha(|f-f_{n}|\chi_{E_{n}})+\frac{\varepsilon}{2}\nonumber \\
& \leq \|f-f_{n}\|_{\infty}\alpha(\chi_{E_{n}})+\frac{\varepsilon}{2}\tag{by Part
(2) of Lemma \ref{property of alpha}}\\
& \leq2\alpha(\chi_{E_{n}})+\frac{\varepsilon}{2}.\tag{$f, f_n\in \mathbb B$ for all $n\geq 1$}
\end{align}
Hence $\alpha \left(  f_{n}-f\right)  \rightarrow0$ as $n\rightarrow \infty.$
Therefore $\alpha$-convergence is equivalent to convergence in measure on
$\mathbb{B}$. Since $\alpha$ was arbitrary, the same holds for $\left \Vert
{}\right \Vert _{2}$-convergence.

2. Suppose $\{g_{n}\}$ is a sequence in $\mathbb{B}$ and $g\in L^{\infty}$
such that $\alpha(g_{n}-g)\rightarrow0.$ Since $\| \|_{1}\leq \alpha(),$ it
follows that $\|g_{n}-g\|_{1}\rightarrow0,$ which implies that $g_{n}%
\rightarrow g$ in measure. Then there is a subsequence $\{g_{n_{k}}\}$ such
that $g_{n_{k}}\rightarrow g$ a.e. $(m).$ Since $g_{n_{k}}\in \mathbb{B}$ for
each $k\in \mathbb{N},$ we can conclude that $|g|\leq1,$ and hence
$g\in \mathbb{B}.$ This completes the proof.
\end{proof}

\section{The Main Result}

In this section we prove our generalization of the classical
Beurling-Helson-Lowdenslager theorem. Suppose $\alpha$ is a continuous
$\Vert \Vert_{1}$-dominating normalized gauge norm. We define $H^{\alpha}$ to
be the $\alpha$-closure of $H^{\infty}$, i.e.,%

\[
H^{\alpha}=\left[  H^{\infty}\right]  ^{-\alpha}.
\]
Since the polynomials in the unit ball $\mathbb{B}$ of $H^{\infty}$ are
$\left \Vert {}\right \Vert _{2}$-dense in $\mathbb{B}$ (consider the Cesaro
means of the sequence of partial sums of the power series), we know from Lemma
\ref{=top} that $H^{\alpha}$ is the $\alpha$-closure of the set of
polynomials.

The following lemma extends the classical characterization of $H^{p}=H^{1}\cap
L^{p}$ to a more general setting.

\begin{lemma}
\label{intersection}Let $\alpha$ be a continuous $\Vert \Vert_{1}$-dominating
normalized gauge norm. Then
\[
H^{\alpha} =H^{1}\cap L^{\alpha}.
\]

\end{lemma}

\begin{proof}
Since $\alpha$ is $\| \|_{1}$-dominating, $\alpha$-convergence implies
$\| \|_{1}$-convergence, thus $H^{\alpha}=[H^{\infty}]^{-\alpha}\subset
[H^{\infty}]^{-\| \|_{1}}=H^{1}.$ Also, $H^{\alpha}=[H^{\infty}]^{-\alpha
}\subset[L^{\infty}]^{-\alpha}=L^{\alpha},$ thus $H^{\alpha}\subset H^{1}\cap
L^{\alpha}.$ Now we suppose $0\neq f\in H^{1}$ $\cap L^{\alpha}$ and
$\varphi \in \left(  L^{\alpha}\right)  ^{\#}$ with $\varphi|H^{\alpha}$ $=0.$
It follows from Proposition \ref{dualspace} that there is an $h\in
\mathcal{L}^{\alpha^{\prime}}$ such that $\varphi \left(  v\right)
=\int_{\mathbb{T}}vhdm$ for all $f\in L^{\alpha}.$ From part (4) of Lemma
\ref{property of alpha} and Lemma \ref{property of alpha prime}, we see
$h\in \mathcal{L}^{\alpha^{\prime}}$ $\subset L^{1}$, so we can write $h\left(
z\right)  =\sum_{n=-\infty}^{\infty}c_{n}z^{n}$. Since $\varphi|H^{\alpha}$
$=0,$ we have%
\[
c_{-n}=\int_{\mathbb{T}}hz^{n}dm=\varphi(z^{n})=0
\]
for all $n\geq0.$ Thus $h$ is analytic and $h\left(  0\right)  =0.$ The fact
that $h\in \mathcal{L}^{\alpha^{\prime}}$ and $f\in L^{\alpha}$ $\cap H^{1}$
imply that $fh$ is analytic and $fh\in L^{1}$, which means $fh\in H^{1}$ $.$
Hence%
\[
\varphi \left(  f\right)  =\int_{\mathbb{T}}fhdm=f\left(  0\right)  h\left(
0\right)  =0.
\]
Since $\varphi|H^{\alpha}$ $=0$ and $f\neq0,$ it follows from the Hahn Banach
theorem that $f\in H^{\alpha},$ which implies $H^{1}$ $\cap L^{\alpha}$
$\subset H^{\alpha}.$ Therefore $H^{1}$ $\cap L^{\alpha}= H^{\alpha}.$
\end{proof}

A key ingredient is based on the following result that uses the Herglotz
kernel \cite{Duren}.

\begin{lemma}
\label{herglotz}$\left \{  \left \vert g\right \vert :0\neq g\in H^{1}\right \}
=\left \{  \varphi \in L^{1}:\varphi \geq0\text{ and }\log \varphi \in
L^{1}\right \}  $. In fact, if $\varphi \geq0$ and $\varphi,\ \log \varphi \in
L^{1}$, then
\[
g\left(  z\right)  =\exp \int_{\mathbb{T}}\frac{w+z}{w-z}\log \varphi \left(
w\right)  dm\left(  w\right)
\]
defines an outer function $g$ on $\mathbb{D}$ and $\left \vert g\right \vert
=\varphi$ on $\mathbb{T}$.
\end{lemma}

Combining Lemma \ref{intersection} and Lemma \ref{herglotz}, we obtain the
following factorization theorem.

\begin{proposition}
\label{decomposition in L^alpha}Let $\alpha$ be a continuous $\Vert \Vert_{1}%
$-dominating normalized gauge norm. If $k\in L^{\infty}$ and $k^{-1}\in
L^{\alpha}$, then there is a unimodular function $w\in L^{\infty}$ and an
outer function $h\in H^{\infty}$ such that $k=wh$ and $h^{-1}\in H^{\alpha}$
$.$
\end{proposition}

\begin{proof}
Recall that an outer function is uniquely determined by its absolute boundary
values, which are necessarily absolutely log integrable. Suppose $k\in
L^{\infty}$ with $k^{-1}\in L^{\alpha}.$ Observe that on the unit circle,
\[
-|k|\leq-\log|k|=\log|k^{-1}|\leq|k^{-1}|,
\]
it follows from $k\in L^{\infty}$ and $k^{-1}\in L^{\alpha}\subset L^{1}$
that
\[
-\infty<-\int_{\mathbb{T}}|k|dm\leq \int_{\mathbb{T}}\log|k^{-1}|dm\leq
\int_{\mathbb{T}}|k^{-1}|dm<\infty,
\]
and hence $|k^{-1}|$ is log integrable. Then by Lemma \ref{herglotz}, there is
an outer function $g\in H^{1}$ such that $|g|=|k^{-1}|$ on $\mathbb{T}.$ If we
let $h=g^{-1}$, then let $w=kg.$ Since $g$ is outer, $h=g^{-1}$ is analytic on
$\mathbb{D}.$ Also, on the unit circle $\mathbb{T},$ $|h|=|g^{-1}|=|k|\in
L^{\infty},$ so $h\in H^{\infty}.$ Moreover, $|g|=|k^{-1}|$ implies $w=kg$ is
unimodular. Hence $k=wh$ where $w$ is unimodular and $h\in H^{\infty}.$
Furthermore, it follows from Lemma \ref{intersection} that $h^{-1}%
=g=wk^{-1}\in L^{\alpha}\cap H^{1}=H^{\alpha}.$
\end{proof}

Before we state our main results, we need the following  lemmas.

\begin{lemma}
\label{weakly closed} Suppose $X$ is a Banach space and $M$ is a closed linear
subspace of $X.$ Then $M$ is weakly closed.
\end{lemma}

\begin{lemma}
\label{translation lemma} Let $\alpha$ be a continuous rotationally symmetric norm
. If $M$ is a closed subspace of $H^{\alpha}  $ invariant under $M_{z},$ which means $zM\subset M,$ then
$H^{\infty} M\subset M.$
\end{lemma}

\begin{proof}
Let ${\mathcal P}_{+}=\{e_n: n\in\mathbb N\}$ denote the class of all polynomials in $H^\infty,$ where $e_n(z)=z^n$ for all $z$ in the unit circle $\mathbb T.$
Since $zM\subset M,$ we see $P(z)M\subset M$  for any polynomial $P\in{\mathcal{P}}_{+}.$
To complete the proof, it suffices to show that  $fh\in M$ for every $ h\in M$ and every $ f\in  H^{\infty}.$  Now we assume that $u$ is a nonzero element in
$ \mathcal L^{\alpha^\prime},$ then
it follows from Proposition \ref{dualspace} that  $hu\in M {\mathcal L^{\alpha^\prime}}\subset L^\alpha (L^\alpha)^{\sharp} \subset L^1.$ Since $f\in H^\infty,$
we obtain  $\hat{f}(n)=\int_{\mathbb T}f(z)z^{-n}dm(z)=0$ for all $n<0,$ which implies that the  partial sums
$S_{n}(f)=\sum_{k=-n}^{n}\hat{f}(n)e_{n}=\sum_{k=0}^{n}\hat{f}(n)e_{n}%
\in{\mathcal{P}}_{+}$ for all $n<0.$ Hence the Cesaro means
$$\sigma_{n}(f)=\frac{S_{0}(f)+S_{1}(f)+\ldots+S_{n}(f)}{n+1}\in \mathcal{P}%
_{+}.$$ Moreover, we know that $\sigma_{n}(f)\rightarrow f$ in the
weak* topology.  Since $hu\in L^1,$ we have
\[
\int_{\mathbb{T}}\sigma_{n}(f)hudm\rightarrow \int_{\mathbb{T}}fhudm.
\]
Observe that $\sigma_{n}(f)h\in{\mathcal{P}_{+}}M\subset M\subset
L^{\alpha}({\mathbb{T}})$ and $u\in{\mathcal{L}}^{\alpha^{\prime}}%
(\mathbb{T}),$ it follows that $\sigma_{n}(f)h\rightarrow fh$ weakly. Since $M$ is a closed subspace of $H^\alpha(\mathbb T),$
 by Lemma \ref{weakly closed}, we see that
$M$ is weakly closed, which means $fh\in M.$  This
completes the proof.
\end{proof}

The following Lemma is the Krein-Smulian theorem.

\begin{lemma}\label{Krein Smulian theorem}
Let $X$ be a Banach space. A convex set in $X^{\sharp}$ is weak* closed if and
only if its intersection with $\{ \phi \in X^{\sharp}: \| \phi \| \leq1\}$ is
weak* closed.
\end{lemma}

The following theorem gives us a most general version of invariant subspace structure.

\begin{theorem}
\label{Saito}Suppose $\alpha$ is a continuous $\Vert \Vert_{1}$-dominating
normalized gauge norm. Let $W$ be an $\alpha$-closed linear subspace of
$L^{\alpha}$, and $M$ be a weak*-closed linear subspace of $L^{\infty}$ such
that $zM\subseteq M$ and $zW\subseteq W$. Then

\begin{enumerate}
\item $M=[M]^{-\alpha}\cap L^{\infty};$

\item $W \cap M$ is weak*-closed in $L^{\infty};$

\item $W=[ W\cap L^{\infty}] ^{-\alpha}.$
\end{enumerate}
\end{theorem}

\begin{proof}
1. It is clear that $M\subset[M]^{-\alpha}\cap L^{\infty}$. Assume, via
contradiction, that $w\in[M]^{-\alpha}\cap L^{\infty}$ and $w\notin M$. Since
$M$ is weak*-closed, there is an $F\in L^{1}$ such that $\int_{\mathbb{T}%
}Fwdm\neq0$ but $\int gFdm=0$ for every $g\in M.$ Since $k=\frac{1}{\left \vert
F\right \vert +1}\in L^{\infty}$ and $k^{-1}\in L^{1}$, it follows from Lemma
\ref{decomposition in L^alpha} that there is an $h\in H^{\infty}$, $1/h\in
H^{1}$ and a unimodular function $u$ such that $k=uh.$ Choose a sequence
$\left \{  h_{n}\right \}  $ in $H^{\infty}$ such that $\left \Vert
h_{n}-1/h\right \Vert _{1}\rightarrow0.$ Since $hF=\bar{u}k F =\bar{u}\frac
{F}{|F|+1}\in L^{\infty},$ we can conclude that
\[
\left \Vert h_{n}hF-F\right \Vert _{1}=\|h_{n}hF-\frac{1}{h}hF\|_1 \leq \left \Vert
h_{n}-\frac{1}{h}\right \Vert _{1}\left \Vert hF\right \Vert _{\infty}%
\rightarrow0.
\]
For each $n\in \mathbb{N}$ and every $g\in M,$ from Lemma \ref{translation lemma} we know that $h_{n}hg\in
H^{\infty}M\subset M$. Hence%
\[
\int_{\mathbb{T}}gh_{n}hF dm=\int_{\mathbb{T}}h_{n}hgFdm=0, \qquad \forall g\in
M.
\]
Suppose $g\in[M]^{-\alpha}.$ Then there is a sequence $\{g_{m}\}$ in $M$ such
that $\alpha(g_{m}-g)\rightarrow0$ as $m\rightarrow \infty.$ For each
$n\in \mathbb{N},$ it follows from $h_{n}hF\in H^{\infty}L^{\infty}\subset
L^{\infty}$ that%

\begin{align}
|\int_{\mathbb{T}}gh_{n}hFdm-\int_{\mathbb{T}}g_{m}h_{n}hFdm| &  \leq
\int_{\mathbb{T}}|(g-g_{m})h_{n}hF|dm\nonumber \\
&  \leq \Vert h_{n}hF\Vert_{\infty}\int_{\mathbb{T}}|g_{m}-g|dm\nonumber \\
&  =\Vert h_{n}hF\Vert_{\infty}\Vert g_{m}-g\Vert_{1}\nonumber \\
&  \leq \Vert h_{n}hF\Vert_{\infty}\alpha(g_{m}-g)\rightarrow0.\tag{$%
m\rightarrow \infty$}%
\end{align}
Thus
\[
\int_{\mathbb{T}}gh_{n}hFdm=\lim_{m\rightarrow \infty}\int_{\mathbb{T}}%
g_{m}h_{n}hFdm=0,\qquad \forall g\in \lbrack M]^{-\alpha}.
\]
In particular, $w\in \lbrack M]^{-\alpha}\cap L^{\infty}$ implies that
$\int_{\mathbb{T}}h_{n}hFwdm=\int_{\mathbb{T}}wh_{n}hFdm=0.$ Hence,
\begin{align*}
0 &  \neq \left \vert \int_{\mathbb{T}}Fwdm\right \vert \\
&  =\lim_{n\rightarrow \infty}\left \vert \int_{\mathbb{T}}Fwdm\right \vert \\
&  \leq \lim_{n\rightarrow \infty}\left \vert \int_{\mathbb{T}}Fw-h_{n}%
hFwdm\right \vert +\lim_{n\rightarrow \infty}\left \vert \int_{\mathbb{T}}%
h_{n}hFwdm\right \vert \\
&  \leq \lim_{n\rightarrow \infty}\left \Vert F-h_{n}hF\right \Vert _{1}\left \Vert
w\right \Vert _{\infty}+0 \\
&  =0,
\end{align*}
a contradiction. Hence $M=[M]^{-\alpha}\cap L^{\infty}$ $.$

2. Let $M=W\cap L^{\infty}.$ To prove $M$ is weak*-closed in $L^{\infty},$
using the Krein-Smulian theorem in Lemma \ref{Krein Smulian theorem}, we need only show that $M\cap \mathbb{B}$ is
weak*-closed. It is clear that $M\cap \mathbb{B}=W\cap \mathbb{B}.$ From part
(2) of Lemma \ref{property of unit ball}, $M\cap \mathbb{B}=W\cap \mathbb{B}$ is
$\alpha$-closed. Since $\alpha$ is continuous, it follows from part (1) of
Lemma \ref{property of unit ball} that $M\cap \mathbb{B}$ is $\left \Vert
{}\right \Vert _{2}$-closed. The fact that $M\cap \mathbb{B}$ is convex implies
$M$ is closed in the weak topology on $L^{2}.$ If $\left \{  f_{\lambda
}\right \}  $ is a net in $\mathbb{B}$ and  $f_{\lambda}\rightarrow f$ weak* in
$L^{\infty}$, then, for every $g\in L^{1},$ $\int_{T}\left(  f_{\lambda
}-f\right)  gdm\rightarrow0.$  But $L^{2}\subset L^{1}$, so $f_{\lambda
}\rightarrow f$ weakly in $L^{2}$. Hence $M\cap \mathbb{B}$ is weak*-closed in
$L^{\infty}.$

3. Since $W$ is $\alpha$-closed in $L^{\alpha},$ it is clear that
$W\supset \left[  W\cap L^{\infty}\right]  ^{-\alpha}.$ Suppose $f\in W$ and
let $k=\frac{1}{\left \vert f\right \vert +1}.$ Then $k\in L^{\infty}$ and
$k^{-1}\in L^{\alpha}.$ It follows from Lemma \ref{decomposition in L^alpha}
that there is an $h\in H^{\infty}$, $1/h\in H^{\alpha}$ and an unimodular
function $u$ such that $k=uh,$ so $hf=\bar{u}kf=\bar{u}\frac{f}{|f|+1}\in
L^{\infty}$ . There is a sequence $\left \{  h_{n}\right \}  $ in $H^{\infty}$
such that $\alpha \left(  h_{n}-\frac{1}{h}\right)  \rightarrow0$. For each
$n\in \mathbb{N},$ it follows from Lemma \ref{translation lemma} that $h_{n}hf\in H^{\infty}H^{\infty}W\subset W$ and
$h_{n}hf\in H^{\infty}L^{\infty}\subset L^{\infty},$ which implies that
$\left \{  h_{n}hf\right \}  $ is a sequence in $W\cap L^{\infty}.$ From part
(2) of Lemma \ref{property of alpha},
\[
\alpha \left(  h_{n}hf-f\right)  \leq \alpha \left(  h_{n}-\frac{1}{h}\right)
\left \Vert hf\right \Vert _{\infty}\rightarrow0.
\]
Thus $f\in \left[  W\cap L^{\infty}\right]  ^{-\alpha}$. Therefore $W=[W\cap
L^{\infty}]^{-\alpha}.$
\end{proof}

A quick corollary of the preceding result is the following conclusion, which
gives us a quick route from the $\left \Vert {}\right \Vert _{2}$-version of
invariant subspace structure  to the (weak*-closed) $\left \Vert {}\right \Vert
_{\infty}$-version of invariant subspace structure.

\begin{corollary}
\label{Lzz}A weak*-closed linear subspace $M$ of $L^{\infty}$ satisfies
$zM\subset M$ if and only if $M=\phi H^{\infty}$ for some unimodular function
$\varphi$ or $M=\chi_{E}L^{\infty}$ for some Borel subset $E$ of $\mathbb{T}$.
\end{corollary}

\begin{proof}
Since $zM\subset M,$ it is easy to check that $[M]^{-\left \Vert {}\right \Vert
_{2}}$ satisfies $z [M]^{-\left \Vert {}\right \Vert _{2}} \subset
[M]^{-\left \Vert {}\right \Vert _{2}} .$ Hence the Beurling-Helson-Lowdenslager
theorem for $\left \Vert {}\right \Vert _{2}$ (Theorem \ref{BHL}) implies
$[M]^{-\left \Vert {}\right \Vert _{2}}=$ $\varphi H^{2}$ for some unimodular
function $\varphi$ or $[M]^{-\left \Vert {}\right \Vert _{2}}=\chi_{E}L^{2}$ for
some Borel subset $E$ of $\mathbb{T}$. It follows from part (1) of Theorem
\ref{Saito} that $M= [M]^{-\left \Vert {}\right \Vert _{2}}\cap L^{\infty}$
equals $\varphi H^{2}\cap L^{\infty}=\varphi H^{\infty}$ or $\chi_{E}L^{2}\cap
L^{\infty}=\chi_{E}L^{\infty}.$
\end{proof}

The following theorem is the generalized Beurling-Helson-Lowdenslager theorem
for a continuous $\| \|_{1}$-dominating normalized gauge norm.

\begin{theorem}
\label{invariantthm} Suppose $\alpha$ is a continuous $\Vert \Vert_{1}%
$-dominating normalized gauge norm and $W$ is a closed subspace of $L^{\alpha
}$. Then $zW\subseteq W$ if and only if either $W=\phi H^{\alpha}$ for some
unimodular function $\phi$ or $W=\chi_{E}L^{\alpha}$ for some Borel set
$E\subset \mathbb{T}$. If $0\neq W\subset H^{\alpha}$, then $W=\varphi
H^{\alpha}$ for some inner function $\varphi$.
\end{theorem}

\begin{proof}
The ``only if" part is obvious. Let $M=W\cap L^{\infty
}$. It follows from part (2) of Theorem \ref{Saito} that $M$ is weak*-closed
in $L^{\infty}.$ Since $zW\subset W,$ it is easy to check that $zM\subset M.$
Then by Corollary \ref{Lzz}, we can conclude that $M=\phi H^{\infty}$ for some
unimodular function $\varphi$ or $M=\chi_{E}L^{\infty}$ for some Borel subset
$E$ of $\mathbb{T}$. It follows from part (3) of Theorem \ref{Saito} that
$W=[W\cap L^{\infty}]^{-\alpha}=[M]^{-\alpha},$ so $W=[\varphi H^{\infty
}]^{-\alpha}=\varphi H^{\alpha}$ for some
unimodular function $\phi$ or $W=[\chi_{E}L^{\infty}]^{-\alpha}=\chi
_{E}L^{\alpha}$ for some inner function $\varphi$. If $0\neq W\subset H^{\alpha}$, we must have $W=\varphi
H^{\alpha}$, so the unimodular function $\varphi=\varphi \cdot1\in \varphi
H^{\alpha}=W\subset H^{\alpha},$ which implies $\varphi$ is an inner function.
\end{proof}

\begin{remark} In Theorem \ref{invariantthm}, there are two situations of the invariant subspace structure as follows:
 \begin{enumerate}\item If $W=\chi
_{E}L^{\alpha},$ then $zW=W,$ which means that $W$ is a doubly invariant subspace in $L^\alpha.$

\item If $W=\varphi H^\alpha,$ then $zW\subsetneqq W.$ In fact, since the multiplication operator $M_\varphi$ is an isometry on $L^\alpha$ and $zH^\alpha\subsetneqq H^\alpha,$ we see $\varphi  zH^\alpha\subsetneqq \varphi H^\alpha,$ which means $zW=\varphi zH^\alpha\subsetneqq \varphi H^\alpha=W.$ This means that $W$ is a simply invariant subspace in $L^\alpha.$
\end{enumerate}\end{remark}

\section{Acknowledgements}

The author is grateful to her advisor,  Don Hadwin,  and to Professor Eric Nordgren, who introduced her to Hardy spaces. They both provided help and support.

\bigskip

\end{document}